\documentclass{compositio}

\usepackage{amsmath,amsthm,amssymb,mathdots,enumerate,hyperref}
\usepackage{stmaryrd}
\usepackage{xcolor}

\usepackage{faktor}
\usepackage{array}
\usepackage{tikz}
\usepackage{tikz-cd}

\usepackage{setspace}
\allowdisplaybreaks

\newtheorem{theorem}{Theorem}[section]
\newtheorem{proposition}[theorem]{Proposition}
\newtheorem{definition}[theorem]{Definition}
\newtheorem{lemma}[theorem]{Lemma}
\newtheorem{corollary}[theorem]{Corollary}
\newtheorem{conjecture}[theorem]{Conjecture}

\theoremstyle{remark}
\newtheorem{remark}[theorem]{Remark}
\newtheorem{example}[theorem]{Example}

\numberwithin{equation}{section}

 \usetikzlibrary{decorations.pathmorphing} 

\DeclareFontEncoding{OT2}{}{}
\DeclareFontSubstitution{OT2}{cmr}{m}{n}
\DeclareFontFamily{OT2}{cmr}{\hyphenchar\font45 }
\DeclareFontShape{OT2}{cmr}{m}{n}{<5><6><7><8><9>gen*wncyr<10><10.95><12><14.4><17.28><20.74><24.88>wncyr10}{}
\DeclareFontShape{OT2}{cmr}{b}{n}{<5><6><7><8><9>gen*wncyb<10><10.95><12><14.4><17.28><20.74><24.88>wncyb10}{}
\DeclareMathAlphabet{\mathcyr}{OT2}{cmr}{m}{n}
\DeclareMathAlphabet{\mathcyb}{OT2}{cmr}{b}{n}
\SetMathAlphabet{\mathcyr}{bold}{OT2}{cmr}{b}{n}

\newcommand{\R}{\mathbb R}
\newcommand{\Z}{{\mathbb Z}}

\newcommand{\C}{{\mathbb C}}
\newcommand{\Q}{{\mathbb Q}}
\newcommand{\F}{{\mathbb F}}
\DeclareMathOperator{\wt}{wt}
\DeclareMathOperator{\dep}{dep}

\newcommand{\kk}{{\bf k}}
\newcommand{\h}{\mathfrak H}
\newcommand{\z}{\xi}
\newcommand{\zn}{z}

\newcommand{\cyc}{\mathrm{cyc}}

\begin{document}

\title{Cyclotomic analogues of \\finite multiple zeta values}

\author{Henrik Bachmann}
\email{henrik.bachmann@math.nagoya-u.ac.jp}
\address{Graduate School of Mathematics\\
Nagoya University\\
Nagoya, Aichi 464-8602\\
Japan}

\author{Yoshihiro Takeyama}
\email{takeyama@math.tsukuba.ac.jp}
\address{Department of Mathematics\\
Faculty of Pure and Applied Sciences\\
University of Tsukuba\\
Tsukuba, Ibaraki 305-8571\\
Japan}

\author{Koji Tasaka}
\email{tasaka@ist.aichi-pu.ac.jp}
\address{Department of Information Science and Technology\\
Aichi Prefectural University\\
Nagakute-city, Aichi 480-1198\\ 
Japan
}




\begin{abstract}
We study the values of finite multiple harmonic $q$-series at a primitive root of unity and show that these specialize to the finite multiple zeta value (FMZV) and the symmetric multiple zeta value (SMZV) through an algebraic and analytic operation, respectively. Further, we prove the duality formula for these values, as an example of linear relations, which induce those among FMZVs and SMZVs simultaneously. 
This gives evidence towards a conjecture of Kaneko and Zagier relating FMZVs and SMZVs. Motivated by the above results, we define cyclotomic analogues of FMZVs, which conjecturally generate a vector space of the same dimension as 
that spanned by the finite multiple harmonic $q$-series at a primitive root of unity of sufficiently large degree.
\end{abstract}

\maketitle

\section{Introduction}
The purpose of this paper is to describe a connection between finite and symmetric multiple zeta (star) values. We explicate this connection in terms of a class of $q$-series evaluated at primitive roots of unity. This construction provides new evidence and a re-interpretation of a conjecture due to Kaneko and Zagier, thus relating finite and symmetric multiple zeta (star) values in an explicit and surprising way. 

For an index $\kk=(k_1,\ldots,k_r)\in (\mathbb{Z}_{\ge 1})^r$ with $k_1\ge2$ 
the multiple zeta value and the multiple zeta star value are defined by 
\begin{align*} 
\zeta(\kk) = \zeta(k_1,\ldots,k_r)=\sum_{m_1>\cdots>m_r>0 }\frac{1}{m_1^{k_1}\cdots m_r^{k_r}},\\ \notag
\zeta^\star(\kk) = \zeta^\star(k_1,\ldots,k_r)=\sum_{m_1\ge\cdots\ge m_r>0 }\frac{1}{m_1^{k_1}\cdots m_r^{k_r}}.
\end{align*}
We denote by $\mathcal{Z}$ the $\Q$-vector space spanned by all multiple zeta values.
The space $\mathcal{Z}$ forms a subalgebra of $\R$ over $\Q$ and is the same space with that spanned by multiple zeta star values.

In \cite{KanekoZagier}, Kaneko and Zagier introduce two objects: the \emph{finite multiple zeta values} $\zeta_{\mathcal{A}}(\kk)$ as elements in the $\Q$-algebra 
$\mathcal{A}=(\prod_{p} \mathbb{F}_p) \big/ (\bigoplus_{p} \mathbb{F}_p)$, where $p$ runs over all primes (see Definition \ref{def:fmzv}), and the \emph{symmetric multiple zeta values} $\zeta_{\mathcal{S}}(\kk)$ as elements in the quotient algebra $\mathcal{Z}/\zeta(2)\mathcal{Z}$ (see Definition \ref{def:symMZV}).
They conjecture that the finite multiple zeta values satisfy the same $\Q$-linear relation as the symmetric multiple zeta values and vice versa (see Conjecture \ref{conj:kanekozagier}). 
A few families of $\Q$-linear relations which are satisfied by the finite and the symmetric multiple zeta values simultaneously 
are obtained by the works of Murahara, Saito and Wakabayashi in \cite{Murahara,SaitoWakabayashi}, where the star versions $\zeta_{\mathcal{A}}^\star(\kk)$ and $\zeta_{\mathcal{S}}^\star(\kk)$ are also considered.

In the present paper, we examine for $n\in \Z_{\ge1}$ the values $\zn_{n}(\kk; \zeta_n)$ and $\zn_{n}^\star (\kk; \zeta_n)$ of finite multiple harmonic $q$-series $\zn_{n}(\kk; q)$ and $\zn_{n}^\star (\kk; q) $ evaluated at a primitive $n$-th  root of unity $\zeta_n$ (see Definition \ref{def:zn}).
These objects lie in the cyclotomic field $\Q(\zeta_n)$.
One of the main results on these values in this paper are the following relations with the finite and symmetric multiple zeta (star) values. 

\begin{theorem}\label{main2} 
For any index $\kk \in (\mathbb{Z}_{\ge 1})^r$, we have 
\begin{align*}
( \zn_p (\kk;\zeta_p) \mod \mathfrak{p}_{p})_p = \zeta_{\mathcal{A}} (\kk), \qquad 
( \zn_p^\star (\kk;\zeta_p) \mod \mathfrak{p}_{p})_p = \zeta_{\mathcal{A}}^\star (\kk)\,,
\end{align*}
where $\mathfrak{p}_{p}=(1-\zeta_p)$ is the prime ideal of $\Z[\zeta_p]$ generated by $1-\zeta_p$. 
\end{theorem}

\begin{theorem}\label{main1}
For any index $\kk \in (\mathbb{Z}_{\ge 1})^r$, the limits 
\begin{align*}
\z(\kk)=\lim_{n\to \infty}\zn_n(\kk;e^{2\pi i/n}), \qquad 
\z^{\star}(\kk)=\lim_{n\to \infty}\zn_n^{\star}(\kk;e^{2\pi i/n})
\end{align*}
exist in $\C$ and it holds that 
\begin{align*}
\mathop{\mathrm{Re}}{\xi(\kk)} \equiv \zeta_{\mathcal{S}} (\kk), \qquad 
\mathop{\mathrm{Re}}{\xi^{\star}(\kk)} \equiv \zeta^{\star}_{\mathcal{S}} (\kk)  
\end{align*}
modulo $\zeta(2) \mathcal{Z}$.
\end{theorem}

Theorems \ref{main2} and \ref{main1} can be applied to the study of $\Q$-linear relations among finite and symmetric multiple zeta (star) values.
In fact we prove a sort of duality formula for the values $z_n^{\star}(\kk;\zeta_n)$ (Theorem \ref{main3} below) and give a new proof of the duality formulas for $\zeta^{\star}_{\mathcal{A}}(\kk)$ and $\zeta_{\mathcal{S}}^\star(\kk)$ via Theorems \ref{main2} and \ref{main1}, which were obtained by Hoffman \cite{Hoffman} and Jarossay \cite{J} respectively (see Section 2.4.3).

\begin{theorem}\label{main3}
For any index $\kk$ and any $n$-th primitive root of unity $\zeta_n$, we have
\begin{align*}
\zn_n^\star (\kk;\zeta_n) = (-1)^{{\rm wt}(\kk)+1} \zn_n^\star (\overline{\kk^\vee};\zeta_n),
\end{align*}
where $\overline{\kk^\vee}$ is the reverse of the Hoffman dual $\kk^\vee$ (see Section \ref{sec:duality}).
\end{theorem}

We shall discuss the dimension of the $\Q$-vector space spanned by the values $z_n(\kk;\zeta_n)$ of fixed weight, which is a subspace of the finite dimensional vector space $\Q(\zeta_n)$.
Here the weight of $z_n(k_1,\ldots,k_r;\zeta_n)$ is $k_1+\cdots+k_r$.
By numerical computations one can observe that the number of linearly independent relations over $\Q$ among $z_p(\kk;\zeta_p)$'s of weight $k$ is stable for sufficiently large prime $p$.

Motivated by this observation, we introduce the \emph{cyclotomic analogue of finite multiple zeta value} $Z(\kk)$ and its star version $Z^\star(\kk)$ in the cyclotomic analogue $\mathcal{A}^\cyc$ of the ring $\mathcal{A}$ (see Definition \ref{def:Z}).
The duality formula also holds for $Z^\star(\kk)$ (see Theorem \ref{thm:delta-duality-Zstar}).
After developing algebraic structures of the spaces spanned by all $Z(\kk)$ and by all $Z^\star(\kk)$, we count the number of linearly independent $\Q$-linear relations, which are consequences of the duality formula (see Theorem \ref{thm:double-shuffle} and Remark \ref{rem:dimension}). 

There is a natural projection that sends $Z(\kk)$ to $\zeta_{\mathcal{A}}(\kk)$ and we also expect, that there is a projection which sends $Z(\kk)$ to $\zeta_{S}(\kk)$ modulo $\zeta(2)\mathcal{Z}$ (see Conjecture \ref{conj:mainconj}).
The conjectured equality of the kernels of these two projections gives a reinterpretation of the Kaneko--Zagier conjecture from the cyclotomic analogue point of view. 
We believe that the cyclotomic analogue may give a new perspective of 
and become a tool for analyzing the Kaneko--Zagier conjecture. 

The contents of this paper are as follows.
In Section \ref{sec:finiteqser}, after developing basic properties on the values $\zn_n(\kk;\zeta_n)$ and $\zn_n^\star(\kk;\zeta_n)$, 
we first give the connection to the finite multiple zeta (star) value (Theorem \ref{main2}). After this we discuss the limit $n\rightarrow \infty$ and show the connection to the symmetric multiple zeta (star) value (Theorem \ref{main1}). We also prove our duality formula (Theorem \ref{main3}) in Section \ref{sec:duality}.
In the last section the cyclotomic analogue of finite multiple zeta (star) values is discussed.

\begin{acknowledgements}
The authors would like to thank Masanobu Kaneko for several fruitful comments and suggestions on this project, Andrew Corbett for a careful read of the manuscript and the referees for constructive comments and suggestions. 
This work was partially supported by JSPS KAKENHI Grant Numbers 16F16021, 16H07115 and 26400106.
 Finally the first and the third authors would like to thank the Max Planck Institute for Mathematics for hospitality.
\end{acknowledgements}


\section{Finite multiple harmonic $q$-series at a root of unity}\label{sec:finiteqser}

\subsection{Definitions}\label{sec:definitions}

In this subsection, we define the finite multiple harmonic $q$-series and give some examples of the value of depth one at a primitive root of unity.

We call a tuple of positive integers $\kk=(k_{1}, \ldots , k_{r})$ an \textit{index}.
An index $\kk=(k_{1}, \ldots , k_{r})$ is said to be \textit{admissible} if $k_{1}\ge 2$ or if it is the empty set $\emptyset$.  

For shorter notation we will write a subsequence $k, k, \ldots , k$ of length $a$ in an index as $\{k\}^{a}$. 
When $a=0$ we ignore it. For example $(\{1\}^{0}, 3, \{1\}^{2}, 2, \{1\}^{0}, 4)=(3, 1, 1, 2, 4)$. 

We define the \emph{weight} $\wt(\kk)$ and the \emph{depth} $\dep(\kk)$ 
of an index $\kk=(k_{1}, \ldots , k_{r})$ by 
\begin{align*}
\wt(\kk) = k_1+\dots +k_r, \qquad 
\dep(\kk)=r. 
\end{align*}

With this notation we can define the following $q$-series which will be one of the main objects in this work.
\begin{definition} \label{def:zn}
Let $n\geq 1$ be a natural number and $q$ a complex number satisfying $q^{m}\neq 1$ for $n>m>0$ 
(to ensure the well-definedness).
For an index $\kk=(k_1,\dots,k_r)$ we define
\begin{align*}
\zn_n(\kk;q) = \zn_n(k_1,\dots,k_r;q) =\sum_{n > m_1 > \dots > m_r >0} 
\frac{q^{(k_1-1) m_1} \dots q^{(k_r-1) m_r}}{ [m_1]_q^{k_1} \dots [m_r]_q^{k_r}}
\end{align*}
and 
\begin{align*}
\zn^\star_n(\kk;q) = \zn_n^\star(k_1,\dots,k_r;q) =\sum_{n > m_1 \geq \dots  \geq  m_r >0} 
\frac{q^{(k_1-1) m_1} \dots q^{(k_r-1) m_r}}{ [m_1]_q^{k_1} \dots [m_r]_q^{k_r}}, 
\end{align*}
where $[m]_{q}$ is the $q$-integer 
\begin{align*}
[m]_q= \frac{1-q^m}{1-q}.
\end{align*}
By agreement we set 
$\zn_n(\kk;q)=0$ if $\dep(\kk) \ge n$ and 
$\zn_n(\emptyset;q)=\zn_n^{\star}(\emptyset;q)=1$. 
\end{definition}

The above $q$-series $\zn_n(\kk;q)$ was also studied by Bradley \cite[Definition 4]{Bradley2} (see also \cite{Kawashima}). 
When $\kk$ is admissible, 
the limit $\lim\limits_{n\rightarrow \infty} \zn_n(\kk;q)$ 
converges for $|q|<1$ and it is called a $q$-analogue of multiple zeta values, since it can be shown that $\lim\limits_{q\to 1}\lim\limits_{n\to \infty} \zn_n(\kk;q)=\zeta(\kk)$.
Their algebraic structure as well as the $\Q$-linear relation were studied by many authors 
\cite{Bradley1,IKOO,OhnoOkuda,OhnoOkudaZudilin,OkudaTakeyama,Takeyama1,Takeyama2,Takeyama,Zhao}. 
There are various different $q$-analogue models in the literature and the one corresponding to our Definition \ref{def:zn} is often called the Bradley-Zhao model. 

\begin{remark}\label{rem:product}
Using the standard decomposition
\begin{align}\label{eq:productrule}
\frac{q^{(k_1-1)m}}{[m]_q^{k_1}} 
\frac{q^{(k_2-1)m}}{[m]_q^{k_2}}=
\frac{q^{(k_1+k_2-1)m}}{[m]_q^{k_{1}+k_{2}}}+(1-q)
\frac{q^{(k_1+k_2-2)m}}{[m]_q^{k_{1}+k_{2}-1}} \qquad (m,k_1,k_2 \ge 1), 
\end{align} 
we see that $\zn_n(\kk; q)$ and $\zn_n^{\star}(\kk; q)$ are related to each other 
in the following way: 
\begin{align}
& 
\zn_{n}^{\star}(\kk;q)=\sum_{\mathbf{a}} \zn_{n}(\mathbf{a};q)+
\sum_{\substack{\kk' \\ \wt(\kk')<\wt(\kk)}} c_{\kk, \kk'} (1-q)^{\wt(\kk)-\wt(\kk')}\zn_{n}(\kk';q),  
\label{eq:zn-star-to-nostar} \\ 
& 
\zn_{n}(\kk;q)=\sum_{\mathbf{a}}(-1)^{\mathrm{dep}(\kk)-\mathrm{dep}(\mathbf{a})} \zn_{n}^{\star}(\mathbf{a};q)+
\sum_{\substack{\kk' \\ \wt(\kk')<\wt(\kk)}} \tilde{c}_{\kk, \kk'}(1-q)^{\wt(\kk)-\wt(\kk')}\zn_{n}^{\star}(\kk';q),  
\label{eq:zn-nostar-to-star} 
\end{align}
where the sum $\sum_{\mathbf{a}}$ is over all indices of the form 
$(k_{1} \square k_{2} \square \cdots \square k_{r})$ in which each $\square$ is `$+$' (plus) or `$,$' (comma) and $c_{\kk, \kk'}$ and $\tilde{c}_{\kk, \kk'}$ are integers independent on $n$. 
For example, it holds that 
\begin{align*}
\zn_{n}^{\star}(3,2,1;q)&=\zn_{n}(3,2,1;q)+\zn_{n}(5,1;q)+\zn_{n}(3,3;q)+\zn_{n}(6;q) \\ 
&\quad {}+(1-q)\left( \zn_{n}(4,1;q)+\zn_{n}(3,2;q)+2\zn_{n}(5;q)\right)+(1-q)^2\zn_{n}(4;q).
\end{align*}
Moreover, using \eqref{eq:productrule}, we can write the product $z_{n}(\kk; q)z_{n}(\kk'; q)$ as a $\Q$-linear combination of $(1-q)^{\mathrm{wt}(\kk)+\mathrm{wt}(\kk')-\mathrm{wt}(\kk'')}
z(\kk''; q)$ with indices $\kk''$ satisfying $0\le \mathrm{wt}(\kk'')\le \mathrm{wt}(\kk)+\mathrm{wt}(\kk')$. 
For example, we see that 
\begin{align*}
z_{n}(1;q)z_{n}(2;q)&=
\left(\sum_{n>m>l>0}+\sum_{n>l>m>0}+\sum_{n>m=l>0}\right)\frac{1}{[m]}\frac{q^{l}}{[l]^{2}} \\ 
&=z_{n}(1,2; q)+z_{n}(2,1; q)+z_{n}(3;q)+(1-q)z_{n}(2;q). 
\end{align*}
\end{remark}

We mainly consider the values $\zn_{n}(\kk; q)$ and $\zn_{n}^{\star}(\kk; q)$ where $q$ is equal to a primitive $n$-th root of unity $\zeta_{n}$. 
They are well-defined as the elements in the cyclotomic field $\Q(\zeta_n)$. 
For example, the first few values $\zn_n(k;\zeta_n)$ of depth one are given by
\begin{equation}\label{eq:example-zn-depth1}
\begin{aligned}
&\zn_n(1;\zeta_n) =\frac{n-1}{2} (1-\zeta_n)\,,\quad \zn_n(2;\zeta_n)=-\frac{n^2-1}{12} (1-\zeta_n)^2\,,\\
&\zn_n(3;\zeta_n) =\frac{n^2-1}{24} (1-\zeta_n)^3 \,, \quad
\zn_n(4;\zeta_n)=\frac{(n^2-1)(n^2-19)}{720} (1-\zeta_n)^4. 
\end{aligned}
\end{equation}
These can be deduced in general from the following formula for the generating function
\begin{align}
\sum_{k>0} \zn_n(k;\zeta_n) \left( \frac{x}{1-\zeta_n}  \right)^k = \frac{n x}{1-(1+x)^n} + 1,
\label{eq:generating-function-depth-one-pre}
\end{align}
which can be shown by using the basic properties of the $n$-th root of unity $\zeta_n$. In particular this shows that $\zn_n(k;\zeta_n) \in (1-\zeta_n)^k \cdot \Q$.

\begin{remark} \label{rem:rationality}
The formula \eqref{eq:generating-function-depth-one-pre} implies that for $k\ge1$ 
\begin{equation}\label{eq:degbern}
\frac{\zn_n(k;\zeta_n)}{(n(1-\zeta_n))^{k}} = -\frac{\beta_k(n^{-1})}{k!} \,,
\end{equation}
where $\beta_k(x) \in \Q[x]$ is the \emph{degenerate Bernoulli number} defined by Carlitz in \cite{C}.  
Since the limit of $\beta_k(n^{-1})$ as $n \to \infty$ is equal to 
the $k$-th Bernoulli number $B_k$, 
the formula \eqref{eq:degbern} can be viewed as a finite analogue of Euler's formula given by
$\zeta(k)/(-2\pi i)^{k}=-B_{k}/2k!$ for even $k$.
\end{remark}

\subsection{Connection with finite multiple zeta values}\label{subsec:conn-FMZV}
In this subsection, we give a proof of Theorem \ref{main2}.

\subsubsection{Definition of finite multiple zeta values}

The finite multiple zeta values will be elements in the ring
\begin{align*}
\mathcal{A} = \left( \prod_{p:{\rm prime}} \F_p \right) \big/ \left( \bigoplus_{p:{\rm prime}} \F_p \right)\,.
\end{align*}
Its elements are of the form $(a_p)_p$, where $p$ runs over all primes and $a_p\in \F_p$.  
Two elements $(a_p)_p$ and $(b_p)_p$ are identified if and only if 
$a_p=b_p$ for all but finitely many primes $p$.
The ring $\mathcal{A}$, which was introduced by Kontsevich \cite[\S2.2]{Kontsevich}, 
carries a $\Q$-algebra structure by sending $a\in \Q$ to $(a \mod p)_p\in \mathcal{A}$ 
diagonally except for finitely many primes which divide the denominator of $a$. 

\begin{definition} \label{def:fmzv}
For an index $\kk=(k_1,\ldots,k_r)$, we define the \emph{finite multiple zeta value}
\begin{align*}
\zeta_{\mathcal{A}} (\kk)=\zeta_{\mathcal{A}} (k_1,\ldots,k_r) = 
\left( \sum_{p>m_1>\cdots >m_r>0} \frac{1}{m_1^{k_1}\cdots m_r^{k_r}} \mod p \right)_p \in \mathcal{A} 
\end{align*}
and its star version
\begin{align*}
\zeta_{\mathcal{A}}^\star (\kk)=\zeta_{\mathcal{A}}^\star (k_1,\ldots,k_r) = 
\left( \sum_{p>m_1\ge \cdots \ge m_r>0} \frac{1}{m_1^{k_1}\cdots m_r^{k_r}} \mod p \right)_p \in \mathcal{A}. 
\end{align*}
\end{definition}

\subsubsection{Proof of Theorem \ref{main2}}

Now we prove Theorem \ref{main2}, 
which is immediate from the standard facts on the algebraic number theory 
(see, e.g., \cite{Washington}).

\begin{proof}[Proof of Theorem \ref{main2}]
For $p$ prime and any $p$-th primitive root of unity $\zeta_p$, 
the ring $\Z[\zeta_p]$ is the ring of algebraic integers in the cyclotomic field $\Q(\zeta_p)$.  
Since the value $[m]_{\zeta_p}=(1-\zeta_p^m)/(1-\zeta_p)$ is a cyclotomic unit, 
$\zn_n(\kk;\zeta_p)$ and $\zn_n^\star(\kk;\zeta_p)$ belong to $\Z[\zeta_p]$. 

Let $\mathfrak{p}_p=(1-\zeta_p)$ be the prime ideal of $\Z[\zeta_p]$ generated by $1-\zeta_p$. 
Since the norm of $\mathfrak{p}_p$ is equal to $p$, 
we have $\Z[\zeta_p]/\mathfrak{p}_p=\F_p$.
Now Theorem \ref{main2} follows from $[m]_{\zeta_p} \equiv m \mod \mathfrak{p}_p$ for $p>m>0$.
\end{proof}

\subsection{Connection with symmetric multiple zeta values}\label{sec:connectiontosymmzv}
In this subsection, we prove Theorem \ref{main1}.

\subsubsection{Definition of symmetric multiple zeta values}

To define the symmetric multiple zeta values, we recall Hoffman's algebraic setup \cite{Hoffman1} with a slightly different convention.

Let $\mathfrak{H}^1=\Q\langle e_1,e_2,\dots\rangle$ be the noncommutative polynomial algebra of 
indeterminates $e_j$ with $j\geq 1$ and set for an index $\kk=(k_1,\ldots,k_r)$
\begin{align*}
e_\kk:=e_{k_1}\cdots e_{k_r} \,.
\end{align*}
For the empty index $\emptyset$ we set $e_{\emptyset}=1$. 
The monomials $\{e_{\kk}\}$ associated to all indices $\kk$ 
form a basis of $\mathfrak{H}^1$ over $\Q$.  

The \textit{stuffle product} is the $\Q$-bilinear map 
$\ast:\mathfrak{H}^1\times \mathfrak{H}^1\rightarrow \mathfrak{H}^1$ characterized by 
the following properties: 
\begin{equation}\label{eq:stuffle}
\begin{aligned}
& 
1\ast w =w\ast 1=w \quad (w\in \mathfrak{H^1}), \\
&
e_{k}w \ast e_{k'}w' = e_{k} (w \ast e_{k'}w') +e_{k'} (e_{k}w \ast w') + e_{k+k'}(w\ast w') \quad (k,k'\geq 1,w,w'\in\mathfrak{H}^1). 
\end{aligned}
\end{equation}
We denote by $\mathfrak{H}^1_{\ast}$ the commutative $\Q$-algebra $\mathfrak{H}^1 $ 
equipped with the multiplication $\ast$.

As stated in \cite[Proposition 1]{IKZ}, there exists a unique $\Q$-algebra homomorphism 
$R: \h^{1}_\ast \to \R[T]$ satisfying 
$R(1)=1, \, R(e_{1})=T$ and $R(e_{\kk})=\zeta(\kk)$ for any admissible index $\kk$
\footnote{The map $R$ is denoted by $Z^{*}$ in \cite{IKZ}.}. 
For an index $\kk$ we define the \textit{stuffle regularized multiple zeta value} $R_{\kk}(T)$ by 
\begin{align*}
R_{\kk}(T):=R(e_{\kk}) \in \R[T] .
\end{align*}
Note that $R_{\emptyset}(T)=1$ and  $R_\kk(T)=\zeta(\kk)$ if $\kk$ is admissible.

\begin{definition}\label{def:symMZV}
For an index $\kk=(k_1,\ldots,k_r) \in (\mathbb{Z}_{\ge 1})^r$ we define the \emph{symmetric multiple zeta value} 
\begin{align*} \zeta_{\mathcal{S}}(\kk)&=\zeta_{\mathcal{S}}(k_1,\ldots,k_r)=
\sum_{a=0}^r (-1)^{k_1+\dots+k_a} 
R_{k_a,,k_{a-1},\dots,k_1}(T) R_{k_{a+1},k_{a+2},\dots,k_r}(T) . 
\end{align*}
and its star version
\[
\zeta^\star_{\mathcal{S}}(\kk)=\zeta^\star_{\mathcal{S}}(k_1,\ldots,k_r)=
\sum_{\footnotesize \square \ \mbox{is `,' or `$+$'}}
\zeta_{\mathcal{S}}(k_1 \square \cdots \square k_r).
\label{eq:gamma-star-nostar}
\]
\end{definition}

Kaneko and Zagier \cite{KanekoZagier} showed that the symmetric multiple zeta value does not depend on $T$, i.e. we have 
\begin{align}
\zeta_{\mathcal{S}}(k_1,\ldots,k_r)=\sum_{a=0}^{r} (-1)^{k_1+\dots+k_a} 
R_{k_a,\dots,k_1}(0) R_{k_{a+1},\dots,k_r}(0) \in \R.
\label{eq:sym-rem}  
\end{align}
This can be checked from the following lemma, 
which will be also used to compute the limit of $\zn_n(\kk;e^{2\pi i/n})$ as $n \to \infty$ in Theorem \ref{thm:limit-in-terms-of-Z}. 

\begin{lemma}\label{lem:ZZ-indep}
For any index $\kk=(k_{1}, \ldots , k_{r})$, the polynomial 
\begin{align}
\sum_{a=0}^{r}(-1)^{k_1+\cdots+k_a}
R_{k_{a}, k_{a-1}, \ldots , k_{1}}(T+X) R_{k_{a+1}, k_{a+2}, \ldots , k_{r}}(T-X)  
\label{eq:Z-star-symm-pol}
\end{align} 
does not depend on $T$. 
Hence it is equal to 
\begin{align*}
\sum_{a=0}^{r}(-1)^{k_1+\cdots+k_a}
R_{k_{a}, k_{a-1}, \ldots , k_{1}}(X) R_{k_{a+1}, k_{a+2}, \ldots , k_{r}}(-X).  
\end{align*}
\end{lemma}

\begin{proof}
From the definition, we see that the polynomial \eqref{eq:Z-star-symm-pol} is a sum of polynomials of the form 
\begin{align}
 \pm \sum_{a=0}^{s}(-1)^a
R_{\{1\}^a,\kk}(T+X) R_{\{1\}^{s-a},\kk'}(T-X) 
\label{eq:Z-star-symm-pol-comp}
\end{align}
with some admissible indices $\kk$ and $\kk'$. For any index $\kk=(k_{1}, \ldots , k_{r})$ and $s\geq 0$, it holds that 
\begin{align}
e_{1}*(e_{1}^{s}e_{\mathbf{k}})=(s+1)e_{1}^{s+1}e_{\mathbf{k}}+
\sum_{a=1}^{r}(e_{1}^{s}e_{\mathbf{k}'(a)}+e_{1}^{s}e_{\mathbf{k}''(a)})+
\sum_{b=1}^{s}e_{1}^{b-1}e_{2}e_{1}^{s-b}e_{\mathbf{k}}, 
\label{eq:e1-star-prod-reduction}
\end{align} 
where 
\begin{align}
\mathbf{k}'(a)=(k_{1}, \ldots , k_{a}+1, \ldots , k_{r}), \quad 
\mathbf{k}''(a)=(k_{1}, \ldots , k_{a}, 1, k_{a+1}, \ldots , k_{r})\,. 
\label{eq:kk-primes}
\end{align}
Using this one can show by induction on $s$ that
\begin{align}
R_{\{1\}^s,\kk} (T)= \sum_{j=0}^{s} R_{\{1\}^{s-j},\kk}(0) \frac{T^j}{j!}. 
\label{eq:formula-of-reg-MZV} 
\end{align} 
{}From this formula we see that the sum \eqref{eq:Z-star-symm-pol-comp} without sign is equal to 
\begin{align*}
&\sum_{a=0}^{s}\sum_{j=0}^{a}\sum_{l=0}^{s-a}
(-1)^{a} R_{\{1\}^{a-j},\mathbf{k}}(0)R_{\{1\}^{s-a-l}, \mathbf{k'}}(0) \frac{(T+X)^{j}}{j!}\frac{(T-X)^{l}}{l!} \\ 
&=\sum_{\substack{j, l \ge 0 \\ j+l \le s}}
\sum_{a=j}^{s-l}
(-1)^{a} R_{\{1\}^{a-j},\mathbf{k}}(0)R_{\{1\}^{s-a-l}, \mathbf{k'}}(0) \frac{(T+X)^{j}}{j!}\frac{(T-X)^{l}}{l!} \\ 
&=\sum_{m=0}^{s} \sum_{a=0}^{s-m}
(-1)^{a} R_{\{1\}^{a},\mathbf{k}}(0)R_{\{1\}^{s-m-a}, \mathbf{k'}}(0) \sum_{\substack{j+l=m\\j, l \ge 0 }}
(-1)^j\frac{(T+X)^{j}}{j!}\frac{(T-X)^{l}}{l!}
\\
&=\sum_{m=0}^{s}\frac{(-2X)^m}{m!} \sum_{a=0}^{s-m}
(-1)^{a} R_{\{1\}^{a},\mathbf{k}}(0)R_{\{1\}^{s-m-a}, \mathbf{k'}}(0),
\end{align*}
which shows that the polynomial \eqref{eq:Z-star-symm-pol-comp} does not depend on $T$, 
neither does  \eqref{eq:Z-star-symm-pol}.
\end{proof}

\subsubsection{Evaluation of the limit}

In order to evaluate the limit of $\zn_n(\kk;e^{2\pi i/n})$ as $n \to \infty$, we first rewrite the value $\zn_{n}(\kk; e^{2\pi i/n})$. 
Let $n$ be a positive integer. 
When $q=e^{2\pi i/n}$ we see that 
\begin{align*}
\frac{1-q}{1-q^{m}}=e^{-\frac{\pi i}{n}(m-1)}\frac{\sin{\frac{\pi}{n}}}{\sin{\frac{m \pi}{n}}} \qquad 
(n>m > 0).  
\end{align*}
Therefore it holds that 
\begin{align*}
\zn_{n}(\kk;e^{\frac{2\pi i}{n}}) =
\left(e^{\frac{\pi i}{n}}\frac{n}{\pi}\sin{\frac{\pi}{n}}\right)^{\mathrm{wt}(\mathbf{k})}
\sum_{n>m_{1}>\cdots >m_{r}>0}
\prod_{j=1}^{r}\frac{e^{\frac{\pi i}{n}(k_{j}-2)m_{j}}}
{\left(\frac{n}{\pi}\sin{\frac{m_{j}\pi}{n}}\right)^{k_{j}}}
\end{align*}
for any non-empty index $\kk=(k_{1}, \ldots , k_{r})$. 
Decompose the set $\{(m_{1}, \ldots , m_{r}) \in \Z^{r} \, | \, n>m_{1}>\cdots >m_{r}>0\}$ into the disjoint union
\begin{align*}
\bigsqcup_{a=0}^{r}
\{(m_{1}, \ldots , m_{r}) \in \Z^{r} \, | \, n>m_{1}>\cdots >m_{a}>\frac{n}{2} \ge m_{a+1}>\cdots >m_{r}>0 \}
\end{align*}
and change the summation variables $m_{j}$ to $n_{j}=n-m_{a+1-j} \, (1\le j \le a)$ and 
$l_{j}=m_{a+j} \, (1 \le j \le r-a)$. 
Then we find that 
\begin{align*}
& 
\zn_{n}(\kk;e^{\frac{2\pi i}{n}}) =
\left(e^{\frac{\pi i}{n}}\frac{n}{\pi}\sin{\frac{\pi}{n}}\right)^{\mathrm{wt}(\mathbf{k})}\\ 
&\times\sum_{a=0}^{r}(-1)^{\sum_{j=1}^{a}k_{j}} 
\sum_{n/2>n_{1}>\cdots >n_{a}>0}
\prod_{j=1}^{a}\frac{e^{-\frac{\pi i}{n}(k_{a+1-j}-2)n_{j}}}
{\left(\frac{n}{\pi}\sin{\frac{n_{j}\pi}{n}}\right)^{k_{a+1-j}}}
\sum_{n/2\ge l_{1}>\cdots >l_{r-a}>0}
\prod_{j=1}^{r-a}\frac{e^{\frac{\pi i}{n}(k_{a+j}-2)l_{j}}}
{\left(\frac{n}{\pi}\sin{\frac{l_{j}\pi}{n}}\right)^{k_{a+j}}}. 
\end{align*}

Motivated by the above expression we introduce the following numbers. 
For an index $\kk=(k_{1}, \ldots , k_{r})$ and a positive integer $n$, we define 
\begin{align*}
A_{n}^{-}(\kk)&=\sum_{n/2>m_{1}>\cdots >m_{r}>0}
\prod_{j=1}^{r}\frac{e^{-\frac{\pi i}{n}(k_{j}-2)m_{j}}}
{\left(\frac{n}{\pi}\sin{\frac{m_{j}\pi}{n}}\right)^{k_{j}}}, \\ 
A_{n}^{+}(\kk)&=\sum_{n/2 \ge m_{1}>\cdots >m_{r}>0}
\prod_{j=1}^{r}\frac{e^{\frac{\pi i}{n}(k_{j}-2)m_{j}}}
{\left(\frac{n}{\pi}\sin{\frac{m_{j}\pi}{n}}\right)^{k_{j}}}. 
\end{align*}
Then we see that 
\begin{align}\label{eq:zn-in-terms-of-a}
\begin{split}
&
\zn_{n}(\kk;e^{\frac{2\pi i}{n}}) =
\left(e^{\frac{\pi i}{n}}\frac{n}{\pi}\sin{\frac{\pi}{n}}\right)^{\mathrm{wt}(\mathbf{k})}  \\ 
&\qquad {}\times \sum_{a=0}^{r}(-1)^{\sum_{j=1}^{a}k_{j}} 
A_{n}^{-}(k_{a}, k_{a-1}, \ldots , k_{1}) 
A_{n}^{+}(k_{a+1}, k_{a+2}, \ldots , k_{r}).  
\end{split}
\end{align}

In order to evaluate $\zn_{n}(\kk;e^{\frac{2\pi i}{n}})$ as $n \rightarrow \infty$, we now give an asymptotic formula for $A_{n}^{+}(\kk)$. From this the asymptotic formula for $A_{n}^{-}(\kk)$ is obtained, because it is easily seen that
\renewcommand{\arraystretch}{1.5}
\begin{align}\label{eq:a-ina+}
A_{n}^{-}(k_{1}, \ldots , k_{r})=\left\{ 
\begin{array}{ll}
\overline{A_{n}^{+}(k_{1}, \ldots , k_{r})} &  (\hbox{$n$: odd}), \\
\overline{A_{n}^{+}(k_{1}, \ldots , k_{r})}+(-\frac{\pi i}{n})^{k_{1}}\, 
\overline{A_{n}^{+}(k_{2}, \ldots , k_{r})} &  
(\hbox{$n$: even}),
\end{array}
\right. 
\end{align}
where the bar on the right-hand side denotes complex conjugation.  
We begin by giving a formula for $A_{n}^{+}(\kk)$ in the case of an admissible index $\kk$.
\begin{lemma}\label{lem:T-asym-admissible}
Let $\kk$ be an admissible index. 
Then it holds that 
\begin{align*}
A_{n}^{+}(\kk)=\zeta(\kk)+O\left(\frac{(\log{n})^{J_{1}(\kk)}}{n}\right) \qquad (n \to +\infty),   
\end{align*} 
where $J_{1}(\kk)$ is a positive integer which depends on $\kk$. 
\end{lemma}

\begin{proof}
Set $\kk=(k_1,\ldots,k_r)$ and define for $k\geq 1$ the function
\begin{align*}
g_k(x)=e^{(k-2)ix} \left(\frac{x}{\sin x}\right)^k \,. 
\end{align*}
Then it holds that 
$\left|A_{n}^{+}(\kk)-\zeta(\kk)\right| \le I_{1}+I_{2}$, where 
\begin{align*}
I_{1}&=\sum_{n/2\ge m_{1}>\cdots >m_{r}>0}
\prod_{j=1}^{r}\frac{1}{m_{j}^{k_{j}}} 
\left| 
\prod_{j=1}^{r}g_{k_j}\left(\frac{m_{j}\pi}{n}\right)-1
\right|, \\ 
I_{2}&=\sum_{m>n/2}\frac{1}{m^{k_{1}}} 
\left(
\sum_{m>m_{2}>\cdots >m_{r}>0}\prod_{j=2}^{r}\frac{1}{m_{j}^{k_{j}}}  
\right). 
\end{align*} 
Since $g_k(x)=1+(k-2)ix + o(x) \ (x\rightarrow +0)$, 
there exists a positive constant $C$ depending on $k$ such that 
$|g_k(m\pi/n)-1| \le Cm/n$
for all integers $m$ and $n$ satisfying $n/2\ge m>0$.
Using the identity 
\begin{align*}
\left(\prod_{j=1}^{r}x_{j}\right)-1=\sum_{a=1}^{r}\, 
(\prod_{j=1}^{a-1}x_{j}) \, (x_{a}-1)
\end{align*} 
and the inequality $0<(\sin x)^{-1}\le \pi/2x$ on the interval $(0,\frac{\pi}{2}]$, we see that 
\begin{align*}
I_{1}&\le \frac{C_{1}}{n} \sum_{a=1}^{r} \sum_{n/2\ge m_{1}>\cdots >m_{r}>0} 
\frac{1}{m_{1}^{k_{1}} \cdots m_{a}^{k_{a}-1} \cdots m_{r}^{k_{r}}} \\ 
&\le \frac{C_{1}}{n} \sum_{a=1}^{r} \sum_{n/2\ge m_{1}>\cdots >m_{r}>0} 
\frac{1}{m_{1}^{k_{1}-1} m_{2}^{k_{2}} \cdots m_{r}^{k_{r}}} \\ 
&=\frac{C_{1} r}{n} \sum_{n/2\ge m>0}\frac{1}{m^{k_{1}-1}} 
\left(
\sum_{m>m_{2}>\cdots >m_{r}>0}\prod_{j=2}^{r}\frac{1}{m_{j}^{k_{j}}}  
\right) 
\end{align*}
for some positive constant $C_{1}$ which depends on $\kk$.
Using the estimation 
\begin{align*}
\sum_{m>m_{2}>\cdots >m_{r}>0}\prod_{j=2}^{r}\frac{1}{m_{j}^{k_{j}}} \le 
\left(\sum_{s=1}^{m-1}\frac{1}{s}\right)^{r-1} \le 
(2\log{m})^{r-1},  
\end{align*}
we get
\begin{align*}
I_{1}+I_{2} \le C_{2} \left(
\frac{1}{n}\sum_{n/2>m>0}\frac{(\log{m})^{r-1}}{m^{k_{1}-1}}+
\sum_{m>n/2}\frac{(\log{m})^{r-1}}{m^{k_{1}}}
\right)
\end{align*}
for some positive constant $C_{2}$ which depends on $\kk$. 
Since $k_{1} \ge 2$ it holds that 
\begin{align*}
\sum_{n/2>m>0}\frac{(\log{m})^{r-1}}{m^{k_{1}-1}}=O((\log{n})^{r}), \quad 
\sum_{m>n/2}\frac{(\log{m})^{r-1}}{m^{k_{1}}}=O\left(\frac{(\log{n})^{r-1}}{n}\right) 
\end{align*}
as $n \to +\infty$. 
This completes the proof. 
\end{proof}
To compute the asymptotic formula for $A_{n}^{+}(\kk)$ in the case of a non-admissible index $\kk$, we need the following lemma.
\begin{lemma}\label{lem:T-asym-at-one}
We have 
\begin{align*}
A_{n}^{+}(1)=\log{\left(\frac{n}{\pi}\right)}+\gamma-\frac{\pi i}{2}+O\left(\frac{1}{n}\right) \qquad (n \to +\infty)\,,
\end{align*}
where $\gamma$ is Euler's constant.
\end{lemma}
\begin{proof}
From the definition of $A_{n}^{+}(1)$ we see that 
\begin{align*}
A_{n}^{+}(1)=\frac{\pi}{n}\sum_{n/2\ge m>0}\left(
\frac{\cos{\frac{m\pi}{n}}}{\sin{\frac{m\pi}{n}}}-i\right)=
\frac{\pi}{n}\sum_{n/2\ge m>0}
\frac{\cos{\frac{m\pi}{n}}}{\sin{\frac{m\pi}{n}}}-\frac{\pi i}{2}+O\left(\frac{1}{n}\right)
\end{align*} 
as $n \to +\infty$. 
Hence it suffices to show that 
\begin{align}
\frac{\pi}{n}\sum_{n/2\ge m>0}\frac{\cos{\frac{m\pi}{n}}}{\sin{\frac{m\pi}{n}}}=
\log\left(\frac{n}{\pi}\right)+\gamma+O\left(\frac{1}{n}\right) \qquad 
(n \to +\infty).  
\label{eq:T1-asym}
\end{align}

Since the function $f(x)=x^{-1}-(\tan{x})^{-1}$ is positive and increasing 
on the interval $(0,\pi)$, we see that 
\begin{align*}
\int_{0}^{\frac{n-1}{2}}f\Big(\frac{\pi x}{n}\Big)\,dx \le 
\sum_{n/2\ge m>0}\left(\frac{n}{\pi}\frac{1}{m}-\frac{\cos{\frac{m\pi}{n}}}{\sin{\frac{m\pi}{n}}}\right) \le 
\int_{1}^{\frac{n}{2}+1}f\Big(\frac{\pi x}{n}\Big)\,dx.
\end{align*}
Set $g(x)=\log{(1+x)}-\log{(\cos{\frac{\pi x}{2}})}$. 
By direct calculation we have 
\begin{align*}
& 
\int_{0}^{\frac{n-1}{2}}f\Big(\frac{\pi x}{n}\Big)\,dx=\frac{n}{\pi}\left( 
g\Big(-\frac{1}{n}\Big)+\log\left(\frac{\pi}{2}\right)
\right), \\ 
& 
\int_{1}^{\frac{n}{2}+1}f\Big(\frac{\pi x}{n}\Big)\,dx=\frac{n}{\pi}\left( 
g\Big(\frac{2}{n}\Big)+\log{\Big(\frac{n}{\pi}\sin{\frac{\pi}{n}}\Big)}+\log\left(\frac{\pi}{2}\right)
\right)\,.
\end{align*}
Since $g(x)=x+o(x) \,\, (x \to 0)$ and $\log{(x^{-1}\sin{x})}=o(x) \,\, (x \to +0)$, 
there exist positive constants $c_{1}$ and $c_{2}$ such that 
\begin{align*}
\int_{0}^{\frac{n-1}{2}}f\Big(\frac{\pi x}{n}\Big)\,dx \ge -c_{1}+\frac{n}{\pi}\log\left(\frac{\pi}{2}\right), \quad 
\int_{1}^{\frac{n}{2}+1}f\Big(\frac{\pi x}{n}\Big)\,dx \le c_{2}+\frac{n}{\pi}\log\left(\frac{\pi}{2}\right)
\end{align*}
for $n \gg 0$. 
Therefore we find that 
\begin{align*}
\frac{\pi}{n} 
\sum_{n/2\ge m>0}\frac{\cos{\frac{m\pi}{n}}}{\sin{\frac{m\pi}{n}}}=
\sum_{n/2\ge m>0}\frac{1}{m}-\log\left(\frac{\pi}{2}\right)+O\Big(\frac{1}{n}\Big) \qquad (n \to +\infty). 
\end{align*}
Using the asymptotic expansion 
\begin{align*}
\sum_{n/2\ge m>0}\frac{1}{m}=\log\left(\frac{n}{2}\right)+\gamma+O\Big(\frac{1}{n}\Big) \qquad (n \to +\infty),  
\end{align*}
we get the formula \eqref{eq:T1-asym}. 
\end{proof}

We can now compute the asymptotic formula for $A_{n}^{\pm}(\kk)$ for any index.

\begin{proposition}\label{prop:asym-A+}
For any index $\kk$ it holds that 
\begin{align}
A_{n}^{\pm}(\kk)=R_{\kk}\left(\log{\left(\frac{n}{\pi}\right)}+\gamma\mp \frac{\pi i}{2}\right)+
O\left(\frac{(\log{n})^{J(\kk)}}{n}\right) \qquad (n \to +\infty),  
\label{eq:asymptotics-T}
\end{align} 
where $\gamma$ is Euler's constant and $J(\kk)$ is a positive integer which depends on $\kk$. 
\end{proposition} 
\begin{proof}
Let $\kk$ be an admissible index and $s$ a non-negative integer. 
We prove the formula \eqref{eq:asymptotics-T} for $A_{n}^{+}(\{1\}^{s}, \kk)$ by induction on $s$.
The case $s=0$ holds because of Lemma \ref{lem:T-asym-admissible}. 
Assume that the formula for $A_{n}^{+}(\{1\}^{s}, \kk)$ holds for $s>0$.
Now use the identity 
\begin{align*}
\frac{e^{-\frac{\pi i}{n}m}}{\frac{n}{\pi}\sin{\frac{m\pi}{n}}} 
\frac{e^{\frac{\pi i}{n}(k-2)m}}{\left(\frac{n}{\pi}\sin{\frac{m\pi}{n}}\right)^{k}}=
\frac{e^{\frac{\pi i}{n}(k-1)m}}{\left(\frac{n}{\pi}\sin{\frac{m\pi}{n}}\right)^{k+1}}-\frac{2\pi i}{n}
\frac{e^{\frac{\pi i}{n}(k-2)m}}{\left(\frac{n}{\pi}\sin{\frac{m\pi}{n}}\right)^{k}},  
\end{align*} 
for $k \ge 1$ and $n/2\ge m>0$ to obtain 
\begin{align*}
A_{n}^{+}(1)A_{n}^{+}(\{1\}^{s}, \kk)&=(s+1)A_{n}^{+}(\{1\}^{s+1}, \kk) \\ 
&+\sum_{a=1}^{r}\left( 
A_{n}^{+}(\{1\}^{s}, \kk'(a))+A_{n}^{+}(\{1\}^{s}, \kk''(a))-\frac{2\pi i}{n}A_{n}^{+}(\{1\}^{s}, \kk)
\right)\\
&+\sum_{b=1}^{s}\left( 
A_{n}^{+}(\{1\}^{b-1}, 2, \{1\}^{s-b}, \kk)-\frac{2\pi i}{n}A_{n}^{+}(\{1\}^{s}, \kk)
\right),
\end{align*}
where $\kk'(a)$ and $\kk''(a)$ are the indices defined by \eqref{eq:kk-primes}. 
With this the desired formula \eqref{eq:asymptotics-T} can be verified by \eqref{eq:e1-star-prod-reduction} and Lemma \ref{lem:T-asym-at-one}.
Note that the formula for $A^-_n(\kk)$ is then immediate from \eqref{eq:a-ina+}.
\end{proof}

The evaluation of $\zn_{n}(\kk;e^{\frac{2\pi i}{n}})$ for $n\rightarrow \infty$ is now given as follows. 
\begin{theorem}\label{thm:limit-in-terms-of-Z}
For any non-empty index $\kk=(k_{1}, \ldots , k_{r})$ it holds that 
\begin{align*}
& 
\lim_{n \to \infty}\zn_{n}(\kk;e^{\frac{2\pi i}{n}}) =\sum_{a=0}^{r}(-1)^{k_1+\cdots+k_a}
R_{k_{a}, k_{a-1}, \ldots , k_{1}}\Big(\frac{\pi i}{2}\Big)
R_{k_{a+1}, k_{a+2}, \ldots , k_{r}}\Big(-\frac{\pi i}{2}\Big).  
\end{align*} 
\end{theorem}
\begin{proof}
This follows from Lemma \ref{lem:ZZ-indep}, 
Proposition \ref{prop:asym-A+} and \eqref{eq:zn-in-terms-of-a}.
\end{proof}

\begin{remark} As mentioned earlier, there are several different $q$-analogue models of multiple zeta values in the literature and our definition of $\zn_{n}(\kk;q)$ corresponds to the Bradley-Zhao model. For other models an analogue of Theorem \ref{thm:limit-in-terms-of-Z} does not necessarily exist, since for example one can prove the formula
\begin{align*}
\sum_{n > m_1 > m_2 >0} 
\frac{q^{m_1} }{ [m_1]_q  [m_2]_q}\,\Bigr|_{q = e^{\frac{2\pi i}{n}} } = 2  \zeta(2) + 2 \pi i \left( \log\left( \frac{n}{2\pi}  \right) + \gamma  \right) + O\left(\frac{\log{n}}{n}\right) \qquad (n \to +\infty)\, , 
\end{align*} 
which would correspond to the Ohno-Okuda-Zudilin model (\cite{OhnoOkudaZudilin}) for the index $\kk = (1,1)$. 
\end{remark}

\subsubsection{Proof of Theorem \ref{main1}}


For the later purpose we introduce the following complex numbers. 

\begin{definition}\label{def:lim-value-Z}
For a non-empty index $\kk$ we define
\begin{align*}
\z(\kk) =\lim\limits_{n\to \infty} \zn_n(\kk;e^{\frac{2\pi i}{n}}) \quad \hbox{and} \quad 
\z^\star(\kk) =\lim\limits_{n\to \infty} \zn_n^\star(\kk;e^{\frac{2\pi i}{n}})
\end{align*}
and set $\z(\emptyset)=\z^\star(\emptyset)=1$.
\end{definition}
Theorem \ref{thm:limit-in-terms-of-Z} implies that 
\begin{align}
\z(k_1,\ldots,k_r) =\sum_{a=0}^{r}(-1)^{k_1+\cdots+k_a}
R_{k_{a}, k_{a-1}, \ldots , k_{1}}\Big(\frac{\pi i}{2}\Big)
R_{k_{a+1}, k_{a+2}, \ldots , k_{r}}\Big(-\frac{\pi i}{2}\Big),
\label{eq:formula-of-z}  
\end{align}
and
\begin{align}
\z^\star(k_1,\ldots,k_r) =
\sum_{\substack{\footnotesize \square \ \mbox{is `,' or `$+$'}}}
\z(k_1 \square \cdots \square k_r),
\label{eq:formula-of-zs} 
\end{align}
which follows from \eqref{eq:zn-star-to-nostar} and 
$(1-e^{2\pi i/n})^k \zn_n(\kk;e^{2\pi i/n}) \to 0 \, (n \to +\infty)$ for $k>0$.
If $\kk=(k_{1}, \ldots , k_{r})$ is an index with $k_{j} \ge 2$ for all $1\le j \le r$, 
we have the equalities $\z(\kk) = \zeta_{\mathcal{S}}(\kk)$ and $\z^\star(\kk) = \zeta_{\mathcal{S}}^\star (\kk)$ from 
Definition \ref{def:symMZV}, 
and hence $\z(\kk),\z^\star(\kk) \in\R$.

\begin{example}
{}Using \eqref{eq:formula-of-z} one can write down the value $\z(k)$ of depth one: 
\renewcommand{\arraystretch}{1.1}
\begin{align*}
\z(k)=\left\{
\begin{array}{ll} -\pi i &  (k=1) \\
2\zeta(k) & (k\ge 2, \, \hbox{$k$ is even}) \\ 
0 & (k\ge 3, \, \hbox{$k$ is odd})
\end{array}
\right. 
\end{align*}
\renewcommand{\arraystretch}{1}
\end{example}

We are now in a position to prove Theorem \ref{main1}.

\begin{proof}[Proof of Theorem \ref{main1}]
The convergence is already proved.  
{}From \eqref{eq:formula-of-reg-MZV} we see that the coefficient of $T^a$ in the polynomial $R_\kk(T)$ 
lies in $\mathcal{Z}$ for any $a \ge 0$.
Hence the formulas \eqref{eq:sym-rem} and \eqref{eq:formula-of-z} imply that 
$\mathrm{Re}(\z(\kk))-\zeta_{\mathcal{S}}(\kk)$ is a polynomial of $\pi^{2}$ 
whose coefficients belong to $\mathcal{Z}$. 
Therefore $\mathrm{Re}(\z(\kk)) \equiv \zeta_{\mathcal{S}}(\kk)$ modulo $\zeta(2)\mathcal{Z}$. 
The star version is then immediate from \eqref{eq:formula-of-zs}.
\end{proof}


\subsection{Duality formula}\label{sec:duality}
In this subsection, we prove Theorem \ref{main3} and use it to give new proofs of the duality formulas for the finite and the symmetric multiple zeta star values.

\subsubsection{Notation}\label{subsec:index-duality}

For an index $\kk=(k_{1}, \ldots , k_{r})$ we define its \textit{reverse} $\overline{\kk}$ by 
\begin{align*}
\overline{\kk}=(k_{r}, k_{r-1}, \ldots , k_{1}).  
\end{align*}

Let $\tau$ be the automorphism on $\h$ given by $\tau(e_1)=e_0$ and $\tau(e_0)=e_1$. 
Every word $w \in \h^1$ can be written as $w= w^\prime e_1$ with $w^\prime \in \h$. 
Then we set $w^\vee = \tau(w^\prime) e_1 \in \h^1$ and call it the \emph{Hoffman dual} of $w$. 
We also define the Hoffman dual $\kk^{\vee}$ of an index $\kk$ by 
\[ e_{\kk^{\vee}}=(e_{\kk})^{\vee}.\] 
For example, the Hoffman dual of the word $e_3 e_2$ is given by
\begin{align*}
(e_3 e_2)^\vee=(e_0e_0e_1e_0e_1)^\vee = \tau(e_0e_0e_1e_0)e_1 = e_1e_1e_0e_1 e_1 = e_1 e_1 e_2 e_1 \,. 
\end{align*}
Hence $(3, 2)^{\vee}=(1, 1, 2, 1)$. 
Note that $\wt(\kk^{\vee})=\wt(\kk)$ for any index $\kk$.

\subsubsection{Proof of Theorem \ref{main3}}

We will use the following fact. 

\begin{lemma}\label{lem:q-binom-power}
Suppose that $n \ge 1$ and $\zeta_n$ is a primitive $n$-th root of unity. 
Then it holds that $(-1)^{n}\zeta_n^{n(n+1)/2}=-1$.  
\end{lemma}

\begin{proof}[Proof of Theorem \ref{main3}]
Note that any index is uniquely written in the form 
\begin{align}
(\{1\}^{a_{1}-1}, b_{1}+1, \ldots , \{1\}^{a_{r-1}-1}, b_{r-1}+1, \{1\}^{a_{r}-1}, b_{r}),  
\label{eq:index-ab-rep}
\end{align}  
where $r$ and $a_{i}, b_{i} \, (1\le i \le r)$ are positive integers
\footnote{If $r=1$, \eqref{eq:index-ab-rep} should read as $(\{1\}^{a_{1}-1}, b_{1})$.}. 
Denote it by $[a_{1}, \ldots , a_{r}; b_{1}, \ldots , b_{r}]$. 
Then we see that 
\begin{align*}
\overline{[a_{1}, \ldots , a_{r}; b_{1}, \ldots , b_{r}]^{\vee}}=
[b_{r}, \ldots , b_{1}; a_{r}, \ldots , a_{1}]. 
\end{align*}

Now we fix a positive integer $r$ and introduce the generating function 
\begin{align*}
K(x_{1}, \ldots , x_{r}; y_{1}, \ldots , y_{r})=\sum
\frac{z_{n}^{\star}([a_{1}, \ldots , a_{r}; b_{1}, \ldots , b_{r}]; \zeta_n)}{(1-\zeta_n)^{a_1+\cdots+a_r+b_1+\cdots+b_r-1}}
\prod_{i=1}^{r}(x_{i}^{a_{i}-1}y_{i}^{b_{i}-1}),   
\end{align*}
where the sum is taken over all positive integers $a_{i}, b_{i} \, (1\le i \le r)$. 
Then Theorem \ref{main3} follows from the equality 
\begin{align}
K(x_{1}, \ldots , x_{r}; y_{1}, \ldots , y_{r})=K(-y_{r}, \ldots , -y_{1}; -x_{r}, \ldots , -x_{1}). 
\label{eq:zstar-duality-proof1}
\end{align}

Let us prove \eqref{eq:zstar-duality-proof1}. 
It holds that 
\begin{align*}
1+\sum_{a=2}^{\infty}\sum_{B\ge m_{1}\ge \cdots \ge m_{a-1} \ge A}
\frac{x^{a-1}}{\prod_{i=1}^{a-1}(1-\zeta_{n}^{m_{i}})}=
\prod_{i=A}^{B}\frac{1-\zeta_{n}^{i}}{1-x-\zeta_{n}^{i}} 
\end{align*}
for $n>B\ge A>0$, and that  
\begin{align*}
\sum_{b=1}^{\infty}\frac{\zeta_{n}^{bm}}{(1-\zeta_{n}^{m})^{b+1}}y^{b-1}=
\frac{1}{1-\zeta_{n}^{m}}\,\frac{\zeta_{n}^{m}}{1-\zeta_{n}^{m}(1+y)} 
\end{align*}
for $n>m>0$. 
Using the above formulas we have 
\begin{align*}
& 
K(x_{1}, \ldots , x_{r}; y_{1}, \ldots , y_{r})=\sum_{n>l_{1}\ge \cdots \ge l_{r}>0}
\prod_{i=l_{r}}^{n-1}(1-\zeta_n^{i}) \\ 
&\times 
\prod_{j=1}^{r-1}\left(\frac{\zeta_n^{l_{j}}}{1-\zeta_n^{l_{j}}(1+y_{j})}
\prod_{i=l_{j}}^{l_{j-1}}\frac{1}{1-x_{j}-\zeta_n^{i}}\right)
\frac{1}{1-\zeta_n^{l_{r}}(1+y_{r})}\prod_{i=l_{r}}^{l_{r-1}}\frac{1}{1-x_{r}-\zeta_n^{i}},  
\end{align*}
where $l_{0}=n-1$. 
Rewrite the right-hand side above by using the partial fraction expansion 
\begin{align*}\
\prod_{i=A}^{B}\frac{1}{X-\zeta_{n}^{i}}&= 
\sum_{i=A}^{B} \frac{1}{X-\zeta_{n}^i} \prod_{j=A}^{i-1} \frac{1}{\zeta_{n}^i-\zeta_{n}^j} 
\prod_{j=i+1}^{B} \frac{1}{\zeta_{n}^i-\zeta_{n}^j}
\notag \\
&=\sum_{t=A}^{B}\frac{1}{X-\zeta_{n}^{t}}
\frac{(-1)^{B-t}\zeta_{n}^{-\binom{B+1}{2}+At-\binom{t}{2}}}
{\prod_{i=1}^{t-A}(1-\zeta_{n}^{-i})\prod_{i=1}^{B-t}(1-\zeta_{n}^{-i})} 
\end{align*}
for $n>B\ge A>0$. 
Then we find that 
\begin{align*}
&
K(x_{1}, \ldots , x_{r}; y_{1}, \ldots , y_{r}) \\ 
&=\sum_{n>t_{1}\ge l_{1}\ge \cdots \ge t_{r}\ge l_{r}>0}
\prod_{i=l_{r}}^{n-1}(1-\zeta_n^{i}) 
(-1)^{\sum_{j=1}^{r}(l_{j-1}-t_{j})}
\zeta_{n}^{\sum_{j=1}^{r}(-\binom{l_{j-1}+1}{2}+l_{j}t_{j}-\binom{t_{j}}{2})} 
\nonumber \\ 
&\qquad \qquad {}\times 
\prod_{j=1}^{r}\left(\prod_{i=1}^{t_{j}-l_{j}}\frac{1}{1-\zeta_n^{-i}}
\prod_{i=1}^{l_{j-1}-t_{j}}\frac{1}{1-\zeta_n^{-i}}\right) 
\nonumber \\ 
&\qquad \qquad {}\times 
\prod_{j=1}^{r-1}\left(\frac{\zeta_n^{l_{j}}}{1-\zeta_n^{l_{j}}(1+y_{j})}\frac{1}{1-x_{j}-\zeta_n^{t_{j}}}\right)
\frac{1}{1-\zeta_n^{l_{r}}(1+y_{r})}\frac{1}{1-x_{r}-\zeta_n^{t_{r}}}.   
\end{align*}
Now change the summation variable $t_{j}$ and $l_{j}$ to 
$n-l_{r+1-j}$ and $n-t_{r+1-j}$, respectively ($1\le j \le r$). 
As a result we get the desired equality \eqref{eq:zstar-duality-proof1} using Lemma \ref{lem:q-binom-power}.  
\end{proof}

\subsubsection{Duality formula for the finite and symmetric multiple zeta star values}

In \cite[Theorem 4.5]{Hoffman} the reversal relations of the finite multiple zeta (star) values are shown:
\begin{equation}\label{eq:rev-FMZV}
\zeta_{\mathcal{A}}(\kk)=(-1)^{\wt(\kk)}\zeta_{\mathcal{A}}(\overline{\kk}),\ \qquad\zeta_{\mathcal{A}}^\star(\kk)=(-1)^{\wt(\kk)}\zeta_{\mathcal{A}}^\star(\overline{\kk}),
\end{equation}
which are almost immediate from the definition.
We now give a new proof of the duality formula for the finite multiple zeta star value using our results.

\begin{theorem}(Hoffman \cite[Theorems 4.5]{Hoffman})
For any index $\kk$, we have
\[\zeta_{\mathcal{A}}^\star(\mathbf{k})=-\zeta_{\mathcal{A}}^\star(\mathbf{k}^{\vee}).\]
\end{theorem}
\begin{proof}
This is a consequence of Theorem \ref{main3} and \ref{main2} and \eqref{eq:rev-FMZV}.
\end{proof}

We will show the duality formula for the symmetric multiple zeta star value.
To see this, we first note that the values $\xi(\mathbf{k})$ and $\xi^{\star}(\mathbf{k})$ 
have the following properties. 
\begin{theorem}\label{thm:z-duality}
For any index $\kk$, the following relations hold. 
\begin{enumerate}[i)]
 \item $\z(\overline{\kk})=(-1)^{\wt(\kk)}\,\overline{\z(\kk)}, \,\, 
\z^{\star}(\overline{\kk})=(-1)^{\wt(\kk)}\,\overline{\z^{\star}(\kk)}$ 
 \item $\z^{\star}(\kk^{\vee})=-\,\overline{\z^{\star}(\kk)}$
\end{enumerate} 
Here the bar on the right-hand sides denotes complex conjugation. 
\end{theorem}

\begin{proof}
i)\, Changing the summation variable $m_{j}$ to $n-m_{r+1-j} \, (1 \le j \le r)$, 
we see that 
\begin{align*}
\zn_{n}(\overline{\kk}; e^{2\pi i/n})=(-e^{2\pi i/n})^{\wt(\kk)}\, \overline{\zn_{n}(\kk; e^{2\pi i/n})}.    
\end{align*} 
Taking the limit as $n \to +\infty$, we obtain 
$\z(\overline{\kk})=(-1)^{\wt(\kk)}\,\overline{\z(\kk)}$. 
The same calculation works also for $z_{n}^{\star}(\kk; e^{2\pi i/n})$. 

ii) From Theorem \ref{main3} we see that 
$\z^{\star}(\kk)=(-1)^{\wt(\kk)+1}\z^{\star}(\overline{\kk^{\vee}})$. 
Combining it with the equality proved in i), we get the desired equality. 
\end{proof}

We now prove the duality formula for symmetric multiple zeta values, which was also shown in \cite[Corollaire 1.12]{J}.
\begin{corollary}\label{main4}
For any index $\kk$, we have 
\begin{align*}
\zeta^\star_{\mathcal{S}}(\kk) \equiv - \zeta_{\mathcal{S}}^\star (\kk^\vee)
\quad \mbox{and}\quad  
\zeta^\star_{\mathcal{S}}(\kk) \equiv (-1)^{\wt(\kk)}\zeta^\star_{\mathcal{S}}(\overline{\kk}) 
\mod \zeta(2)\mathcal{Z}.
\end{align*}
\end{corollary}
\begin{proof}
This follows directly from Theorems \ref{main1} and \ref{thm:z-duality}.
\end{proof}

\subsection{Example of relations of $z_n^\star(\kk;\zeta_n)$}

In this subsection, using the results obtained in the previous subsections we give an example of relations among $z_n^\star(\kk;\zeta_n)$ and of $\Q$-linear relations among finite and symmetric multiple zeta star values via Theorems \ref{main2} and \ref{main1}, accordingly.

Applying Theorem \ref{main3} to the product $z^{\star}_{n}(\kk; \zeta_{n})z^{\star}_{n}(\kk'; \zeta_{n})$, one has
\begin{align*}
z^{\star}_{n}(\kk; \zeta_{n})z^{\star}_{n}(\kk'; \zeta_{n})=
(-1)^{\wt(\kk)+\wt(\kk')}
z^{\star}_{n}(\overline{\kk^{\vee}}; \zeta_{n})
z^{\star}_{n}(\overline{\kk'^{\vee}}; \zeta_{n}).
\end{align*}
Since each product can be written as $\Q$-linear combinations of $(1-\zeta_{n})^{\wt(\kk)+\wt(\kk')-\wt(\kk'')} z^{\star}_{n}(\kk''; \zeta_{n})$ (see Remark \ref{rem:product}), we can obtain a relation among $z_{n}^{\star}(\kk;\zeta_n)$ over $\Q[1-\zeta_n]$.
As a consequence of the above relations, one can prove for instance the identity
\begin{align}
2\zn^{\star}_{n}(4,1;\zeta_{n})+\zn^{\star}_{n}(3,2;\zeta_{n})=
\frac{(n^4-1)(n+5)}{1440}(1-\zeta_n)^5+\frac{n+2}{3}(1-\zeta_n)^2 \zn_{p}^{\star}(2,1;\zeta_{n})
\label{eq:2_3-prime}
\end{align}
for any $n\ge1$ and any $n$-th primitive root of unity $\zeta_n$.
The identity \eqref{eq:2_3-prime} together with Theorem \ref{main2} shows
\begin{align*}
2\zeta^\star_{\mathcal{A}} (4,1)+ \zeta^\star_{\mathcal{A}}(3,2)=0, 
\end{align*}
which was obtained by Hoffman \cite[Theorem 7.1]{Hoffman}.
On the other hand, using $1-e^{2\pi i/n}=-2\pi i/n+o(1/n)$ as $n \to +\infty$ and Theorem \ref{main1}, we find
\begin{align*}
2\zeta^\star_{\mathcal{S}}(4,1)+\zeta^{\star}_{\mathcal{S}}(3,2) \equiv 0 \mod \zeta(2)\mathcal{Z}\,.
\end{align*}

\section{Cyclotomic analogue of finite multiple zeta values}

\subsection{Definitions}
In this subsection we define the cyclotomic analogue of the finite multiple zeta (star) values $Z(\kk)$ and present its duality formula.
We also compute the value $Z(k)$ of depth one as an example.

As an cyclotomic analogue of the ring $\mathcal{A}$ we define 
\begin{align*}
\mathcal{A}^{\cyc}=\left(\prod_{p:{\rm prime}} \Z[\zeta_p]/(p) \right) 
\bigg/ \left(\bigoplus_{p:{\rm prime}} \Z[\zeta_p]/(p) \right).  
\end{align*}
Similar to $\mathcal{A}$ (see Section \ref{subsec:conn-FMZV}) the ring $\mathcal{A}^{\cyc}$ is a $\Q$-algebra. 

\begin{definition} \label{def:Z}
For an index $\kk$ we define the \emph{cyclotomic analogue of finite multiple zeta value}
\begin{align*}
Z(\kk) = \big(\, \zn_p(\kk;\zeta_p) \mod (p) \,\big)_p \in \mathcal{A}^\cyc, 
\end{align*}
and its star version 
\begin{align*}
Z^\star(\kk) = \big(\, \zn_p^\star(\kk;\zeta_p) \mod (p) \,\big)_p \in \mathcal{A}^\cyc.
\end{align*}
\end{definition}

Recall that $\mathfrak{p}_{p}=(1-\zeta_{p})$ is a prime ideal in $\Z[\zeta_{p}]$ and that $(p)=\mathfrak{p}_{p}^{p-1}$. This yields a surjective map  
\begin{align*}
\Z[\zeta_{p}]/(p) \to \Z[\zeta_{p}]/\mathfrak{p}_{p} \simeq \mathbb{F}_{p}
\end{align*}
for all prime $p$. Let $\varphi$ be the induced $\mathbb{Q}$-algebra homomorphism 
\begin{align}\label{eq:defphi}\begin{split}
\varphi: \mathcal{A}^\cyc &\longrightarrow \mathcal{A}, \\ 
(a_p \mod (p))_p &\longmapsto (a_p \mod \mathfrak{p}_{p})_p\,.
\end{split}
\end{align}
 The map $\varphi$ satisfies $\varphi(Z(\kk))=\zeta_{\mathcal{A}}(\kk)$ and $\varphi(Z^{\star}(\kk))=\zeta^{\star}_{\mathcal{A}}(\kk)$.

Let us write down the formula for $Z(k)$ of depth one.
We write 
\begin{align}
\varpi= (1-\zeta_p)_p \in \mathcal{A}^{cyc}.
\label{eq:def-varpi}
\end{align}
For $k\ge0$ define the numbers $G_{k}$ by 
\begin{align*}
\sum_{k\ge0}G_kz^k = \frac{z}{\log(1+z)}, 
\end{align*}
which are called \emph{Gregory coefficients}. 
It is known that $G_{k}\not=0$ for any $k \ge 0$ (see \cite{Steffensen}).

\begin{proposition}\label{prop:depth-one-formula}
For any $k\ge 1$, we have $Z(k)=-G_k \varpi^k  \in \varpi^k \Q^\times$.
\end{proposition}
\begin{proof}
The generating function \eqref{eq:generating-function-depth-one-pre} can be written as 
\begin{align}
\sum_{k=1}^\infty z_n(k;\zeta_n) \left( \frac{x}{1-\zeta_n}\right)^k=-\sum_{l=1}^{\infty}
\left(-\sum_{j=1}^{\infty}h_{j}(n)x^{j}\right)^{l},
\label{eq:depth-one-pol}
\end{align}
where we let
\begin{align*}
h_{j}(x)=\frac{1}{(j+1)!}\prod_{a=1}^{j}(x-a) \qquad (j \ge 1).  
\end{align*}
Hence, for each $k \ge 1$, 
there exists a unique polynomial $D_{k}(x) \in \Q[x]$ of degree at most $k$ such that 
$z_{n}(k;\zeta_n)=D_{k}(n) (1-\zeta_n)^k$ for all $n \ge 1$.
Then 
\begin{align*}
\zn_p(k;\zeta_p)\equiv D_{k}(0)(1-\zeta_p)^{k}  \mod (p)
\end{align*}
for sufficiently large prime $p$. 
Therefore $Z(k)=D_{k}(0) \, \varpi^k$ for $k \ge 1$.

On the other hand, from \eqref{eq:depth-one-pol} we see that 
\begin{align*}
\sum_{k=1}^{\infty}D_{k}(0)z^{k}=-\sum_{l=1}^{\infty}
\left(-\sum_{j=1}^{\infty}h_{j}(0)x^{j}\right)^{l}=1-\frac{z}{\log{(1+z)}}.  
\end{align*}
Hence $D_{k}(0)=-G_{k}$ for $k \ge 1$, which completes the proof.
\end{proof}

The first few values are given by
\begin{align*}
Z(1)=-\frac{1}{2} \varpi, \quad  Z(2)=\frac{1}{12} \varpi^2, \quad 
Z(3)=-\frac{1}{24} \varpi^3, \quad  Z(4)=\frac{19}{720}\varpi^4,
\end{align*}
which are also obtained from \eqref{eq:example-zn-depth1}.

\subsection{Algebraic structure}

In this subsection, we examine the algebraic structure of the space spanned by $Z(\kk)$'s and $Z^\star(\kk)$'s.
Recall that $\mathfrak{H}^{1}$ is the noncommutative polynomial algebra over $\Q$  of 
indeterminates $e_j$ with $j\geq 1$ and for an index $\kk=(k_1,\ldots,k_r)$ we write $e_\kk:=e_{k_1}\cdots e_{k_r}$.
For simplicity, we introduce the following notation.
Let $\gamma$ be a function defined on the set of indices taking values in a $\Q$-vector space $M$. 
Then, by abuse of notation, we denote by the same letter $\gamma$ 
the $\Q$-linear map $\mathfrak{H}^{1} \to M$ which sends $e_{\kk}$ to $\gamma(\kk)$. 

As a variant of the product $\ast$ given in \eqref{eq:stuffle}, we define the product $\star$ on $\mathfrak{H}^1$ inductively by
\begin{align*}
& 
1\star w =w\star 1=w \quad (w\in \mathfrak{H^1}), \\
&
e_{k}w \star e_{k'}w' = 
e_{k} (w \star e_{k'}w') +e_{k'} (e_{k}w \star w') - e_{k+k'}(w\star w') \quad (k,k' \ge 1, \, w,w'\in\mathfrak{H}^1). 
\end{align*}
Viewing $\zeta$ and $\zeta^\star$ as a map from the set of indices to $\R$, we have for admissible $v,w \in \mathfrak{H}^{1}$
\[\zeta(v \ast w) = \zeta(v)  \zeta(w)\,,\qquad \zeta^\star(v \star w) = \zeta^\star(v) \zeta^\star(w) \,.\]
To describe the algebraic structure of the space spanned by $Z(\kk)$'s and $Z^\star(\kk)$'s we consider the $\mathcal{C}$-module $\widehat{\mathfrak{H}}^{1}=\mathcal{C} \otimes_{\Q} \mathfrak{H}^{1}$, where $\mathcal{C}=\mathbb{Q}[\hbar]$ denotes the polynomial ring of one variable $\hbar$. 
On $\widehat{\mathfrak{H}}^{1}$ we define the $\mathcal{C}$-bilinear maps 
$\ast_{q}, \, \star_{q}: \widehat{\mathfrak{H}}^{1} \times \widehat{\mathfrak{H}}^{1} \to \widehat{\mathfrak{H}}^{1}$ by 
\begin{align*}
& 
1 \ast_{q} w=w \ast_{q} 1=w, \qquad 1 \star_{q} w=w \star_{q} 1=w, \\ 
& 
e_{k_1} v \ast_{q} e_{k_2} w = 
e_{k_1} (v \ast_{q} e_{k_2} w) + e_{k_2} (e_{k_1} v \ast_{q} w) + 
(e_{k_1+k_2}+\hbar \, e_{k_1+k_2-1})(v\ast_{q} w), \\  
& 
e_{k_1} v \star_{q} e_{k_2} w = 
e_{k_1} (v \star_{q} e_{k_2} w)+ e_{k_2} (e_{k_1} v \star_{q} w) - 
(e_{k_1+k_2}+\hbar \, e_{k_1+k_2-1})(v\star_{q} w)
\end{align*}
for $ v,w \in \widehat{\mathfrak{H}}^{1}$ and $k_1,k_2 \geq 1$. 
Similarly as before, for a function $\Gamma$ taking values in a $\mathcal{C}$-module $\widehat{M}$, 
we denote the induced $\mathcal{C}$-linear map $\widehat{\mathfrak{H}}^{1} \to \widehat{M}$ 
by the same letter $\Gamma$. 
For example, $\Gamma(e_2\ast_{q} e_1)=\Gamma(2,1)+\Gamma(1,2)+\Gamma(3)+\hbar \, \Gamma(2)$.

We define the $\Q$-linear action of $\mathcal{C}$ on $\mathcal{A}^{\cyc}$ by 
$\hbar z=\varpi z \, (z \in \mathcal{A}^{\cyc})$, 
where $\varpi$ is given by \eqref{eq:def-varpi}. 
Then the $\mathcal{C}$-linear maps $Z, \, Z^{\star} : \widehat{\mathfrak{H}}^{1} \to \mathcal{A}^{\cyc}$ 
are defined by the properties 
$Z(e_{\kk})=Z(\kk)$ and $Z^{\star}(e_{\kk})=Z^{\star}(\kk)$ for any index $\kk$. 
It follows that they satisfy 
\begin{align}
Z(v \ast_{q} w)=Z(v)Z(w), \qquad 
Z^{\star}(v \star_{q} w)=Z^{\star}(v)Z^{\star}(w)
\label{eq:q-stuffle-hbar}
\end{align}
for any $v, w \in \widehat{\mathfrak{H}}^{1}$ (see \cite[\S2]{Bradley1}). Due to \eqref{eq:q-stuffle-hbar} the product of two $Z(\kk)$ (resp $Z^\star(\kk)$) can be written as a $\Q[\varpi]$-linear combination of $Z(\kk)$ (resp. $Z^\star(\kk)$). In fact, the next lemma shows that these can be written as a $\Q$-linear combination of $Z(\kk)$ (resp. $Z^\star(\kk)$).

\begin{lemma}\label{lem:product}
For any index $\kk$, we have
\begin{align*}
\varpi \, Z(\kk) &= -\frac{2}{2\dep(\kk)+1} \, Z(e_1\ast e_\kk), \\ 
\varpi \, Z^\star(\kk) &= \frac{2}{2\dep(\kk)-1} \, Z^\star(e_1\star e_\kk). 
\end{align*}
\end{lemma}

\begin{proof}
It holds that 
\begin{align*}
e_{1}\ast_{q}e_{\kk}=e_{1}\ast e_{\kk}+\hbar \dep(\kk) e_{\kk}, \quad 
e_{1}\star_{q}e_{\kk}=e_{1}\star e_{\kk}-\hbar \dep(\kk) e_{\kk}
\end{align*}
for any index $\kk$. 
Now the desired formula follows from \eqref{eq:q-stuffle-hbar} and $Z(1)=-\varpi /2$.  
\end{proof}

Motivated by Lemma \ref{lem:product} we define 
the $\Q$-linear maps $L, L^{\star}: \mathfrak{H}^{1} \to \mathfrak{H}^{1}$ by 
\begin{align*}
L(e_{\kk})=-\frac{2}{2\dep(\kk)+1} \, e_1\ast e_\kk, \quad 
L^{\star}(e_{\kk})=\frac{2}{2\dep(\kk)-1} \, e_1\star e_\kk 
\end{align*}
for any index $\kk$. 
Note that if $\wt(\kk)=k$ then $L(e_{\kk})$ and $L^{\star}(e_{\kk})$ are written as a $\Q$-linear combination 
of monomials of weight $k+1$. 
Using these maps we introduce the $\Q$-linear maps 
$\rho, \rho^{\star}: \widehat{\mathfrak{H}}^{1} \to \mathfrak{H}^{1}$ defined by 
\begin{align*}
\rho(\hbar^{k}w)=L^{k}(w), \quad 
\rho^{\star}(\hbar^{k}w)=(L^{\star})^{k}(w) \qquad 
(k\ge 0, \, w \in \mathfrak{H}^{1})\,,
\end{align*}
with $L^{0}(w)=(L^{\star})^{0}(w)=w$. Note that $\rho(v)=v$ for $v \in \mathfrak{H}^{1}$ and by Lemma \ref{lem:product} we get
\begin{align}
Z(\rho(w))=Z(w), \quad Z^{\star}(\rho^{\star}(w))=Z^{\star}(w) \qquad (w \in \widehat{\mathfrak{H}}^{1}).  
\label{eq:q-lift-hbar}
\end{align}
Now define the $\Q$-bilinear maps 
$\tilde{\ast}, \tilde{\star}: \mathfrak{H}^{1} \times \mathfrak{H}^{1} \to \mathfrak{H}^{1}$ by 
\begin{align*}
v \,\tilde{\ast}\, w=\rho(v \ast_{q} w), \quad 
v \,\tilde{\star}\, w=\rho^{\star}(v \star_{q} w) \qquad (v, w \in \mathfrak{H}^{1})
\end{align*}
and define for $d \ge 0$ the space 
\begin{align*}
\mathfrak{H}^{1}_{d}=\bigoplus_{\substack{\kk \\ \wt(\kk)=d}}\Q\, e_{\kk}\,,   
\end{align*}
which is a $\Q$-linear subspace of $\mathfrak{H}^{1}$. 

\begin{proposition}\label{prop:star-prod-cyc}
(i)\, It holds that  
$\mathfrak{H}^{1}_{d_{1}} \,\tilde{\ast}\, \mathfrak{H}^{1}_{d_{2}} \subset \mathfrak{H}^{1}_{d_{1}+d_{2}}$ and 
$\mathfrak{H}^{1}_{d_{1}} \,\tilde{\star}\, \mathfrak{H}^{1}_{d_{2}} \subset \mathfrak{H}^{1}_{d_{1}+d_{2}}$ 
for $d_{1}, d_{2} \ge 0$. 

(ii)\, For $v, w \in \mathfrak{H}^{1}$, it holds that 
$Z(v \,\tilde{\ast}\, w)=Z(v)Z(w)$ and 
$Z^{\star}(v \,\tilde{\star}\, w)=Z^{\star}(v)Z^{\star}(w)$. 
\end{proposition}

\begin{proof}
(i)\, Note that, if we define the weight of $\hbar$ to be one, 
then the $\mathcal{C}$-bilinear maps $\ast_{q}$ and $\star_{q}$ preserve the total weight. 
Hence the statement follows from the property 
$L(\mathfrak{H}^{1}_{d}) \subset \mathfrak{H}^{1}_{d+1}$ and 
$L^{\star}(\mathfrak{H}^{1}_{d}) \subset \mathfrak{H}^{1}_{d+1}$. 

(ii)\, This follows from \eqref{eq:q-stuffle-hbar} and \eqref{eq:q-lift-hbar}. 
\end{proof}

\begin{corollary}\label{cor:product}
For positive integers $k,k'$, let $\kk$ and $\kk'$ be indices of weight $k$ and $k'$.
Then the product $Z(\kk)Z(\kk')$ (resp. $Z^\star(\kk)Z^\star(\kk')$) can be written as 
$\Q$-linear combinations of $Z(\mathbf{a})$'s (resp. $Z^\star(\mathbf{a})$'s) of weight $k+k'$.
\end{corollary}

\subsection{Dimension of the space of $Z(\kk)$}

In this subsection we discuss the dimension of the $\Q$-vector space 
spanned by $Z(\kk)$'s and $Z^{\star}(\kk)$'s. 
First we note the following fact. 

\begin{proposition}
For any $k \ge 0$, it holds that $Z(\mathfrak{H}^{1}_{k})=Z^{\star}(\mathfrak{H}^{1}_{k})$ 
as a $\Q$-linear subspace of $\mathcal{A}^{\cyc}$. 
\end{proposition}

\begin{proof}
{}From \eqref{eq:zn-star-to-nostar} we see that $Z^{\star}(\kk)$ is represented as 
\begin{align*}
Z^{\star}(\kk)=\sum_{\substack{\kk' \\ \wt(\kk') \le \wt(\kk)}} c_{\kk, \kk'} \varpi^{\wt(\kk)-\wt(\kk')} Z(\kk'),    
\end{align*} 
where $c_{\kk, \kk'} \in \Q$. 
Lemma \ref{lem:product} implies that 
$\varpi^{\wt(\kk)-\wt(\kk')} Z(\kk')=Z(L^{\wt(\kk)-\wt(\kk')}(e_{\kk'}))$, and 
the weight of $L^{\wt(\kk)-\wt(\kk')}(e_{\kk'})$ is equal to $\wt(\kk)$. 
Hence $Z^{\star}(\kk) \in Z(\mathfrak{H}^{1}_{k})$ for any index $\kk$ of weight $k$.  
In the same way we see that $Z(\kk) \in Z^{\star}(\mathfrak{H}^{1}_{k})$ 
if $\wt(\kk)=k$ from \eqref{eq:zn-nostar-to-star} and therefore $Z(\mathfrak{H}^{1}_{k})=Z^{\star}(\mathfrak{H}^{1}_{k})$. 
\end{proof}

\begin{theorem}\label{thm:delta-duality-Zstar}
For any index $\kk$ we have
\begin{align}\label{eq:zduality}
Z^{\star}(\mathbf{k})=(-1)^{\mathrm{wt}(\mathbf{k})+1}Z^{\star}(\overline{\mathbf{k}^{\vee}}).   
\end{align} 
\end{theorem}
\begin{proof}
The formula is immediate from Theorem \ref{main3} and Definition \ref{def:Z}.
\end{proof}

Combining Theorem \ref{eq:zduality} with Proposition \ref{prop:star-prod-cyc} (ii),  
we obtain a variant of the double shuffle relation \cite{IKZ} among $Z^{\star}(\kk)$'s.
To describe it, we denote by $\delta$ the $\Q$-linear map 
$\delta : \h^1 \rightarrow \h^1$ sending $e_\kk$ to 
$(-1)^{\wt(\kk)+1}e_{\overline{\kk^{\vee}}}$ for any index $\kk$. 
Note that the map $\delta$ is an involution on $\h^1$ and with this \eqref{eq:zduality} can be stated as $Z^{\star}(e_\kk) = Z^{\star}(\delta(e_{\kk}))$.

\begin{theorem}\label{thm:double-shuffle}
For any indices $\kk$ and $\kk'$, we have
\begin{align*}
Z^\star\Big(e_\kk \,\tilde{\star}\, e_{\kk'}-\delta\big((\delta(e_{\kk}) \,\tilde{\star}\, \delta(e_{\kk'})\big)\Big)=0.
\end{align*} 
\end{theorem}

\begin{proof}
This follows from Proposition \ref{prop:star-prod-cyc} (ii) and 
Theorem \ref{thm:delta-duality-Zstar} because
\begin{align*}
Z^\star\Big(\delta\big(\delta(e_{\kk}) \,\tilde{\star}\, \delta(e_{\kk'})\big)\Big)=
Z^\star\big(\delta(e_{\kk}) \,\tilde{\star}\, \delta(e_{\kk'})\big)=
Z^\star\big(\delta(e_{\kk})\big)Z^\star\big(\delta(e_{\kk'})\big)=Z^\star(e_\kk)Z^\star( e_{\kk'}),  
\end{align*}
which is equal to $Z^{\star}(e_{\kk} \,\tilde{\star}\, e_{\kk'})$. 
\end{proof}

\begin{remark} \label{rem:dimension}
For $k \ge 0$ we define the $\Q$-linear subspace $\mathcal{Z}_{k}^{\cyc}$ of $\mathcal{A}^{\cyc}$ by 
\begin{align*}
\mathcal{Z}_{k}^{\cyc}=Z^{\star}(\mathfrak{H}^{1}_{k})=Z(\mathfrak{H}^{1}_{k}). 
\end{align*}
Using Theorem \ref{thm:delta-duality-Zstar} and \ref{thm:double-shuffle}, we have the following upper bounds 
for the dimension of $\mathcal{Z}^{\cyc}_{k}$:
\begin{center}
$\begin{array}{c|cccccccccccccccccccccccccc}
k & 0&1&2&3&4&5&6&7&8&9&10&11&12\\ \hline
\dim_\Q \mathcal{Z}_k^\text{cyc}  \le & 1&1&1&2&2&4&5&8&12&17&27&38&57\\
\end{array}$
\end{center}

For prime $p \ge 2$ and $k \ge 0$, we denote by $\mathcal{Z}^{(p)}_{k}$ the $\Q$-vector space 
spanned by $z_{p}(\kk; e^{2\pi i/p})$ with $\wt(\kk)=k$. Notice that $\dim_{\Q}\mathcal{Z}^{(p)}_{k} \leq p-1 = [ \Q(\zeta_p):\Q]$. Denote by $d_k$ the numbers in the second column of the above table. By numerical experiments, we observed that for $1\le k \le 12$ we have $\dim_{\Q}\mathcal{Z}^{(p)}_{k} \geq d_k$ for primes $p > d_k$ up to $p=113$. Thus, we might expect that Theorem \ref{thm:delta-duality-Zstar} and \ref{thm:double-shuffle} give all $\Q$-linear relations among the $Z^\star(\kk)$'s.
\end{remark}

\subsection{Kaneko--Zagier conjecture revisited}\label{sec:kanekozagierconj}
In this subsection we will give a new interpretation of the Kaneko--Zagier conjecture in terms of the cyclotomic analogue of finite multiple zeta values $Z(\kk)$. Let us first recall the statement of their conjecture. Let $\mathcal{Z}_{\mathcal{A}}$ be the $\Q$-vector space of finite multiple zeta values.
It forms a $\Q$-algebra. 
\begin{conjecture}(Kaneko--Zagier \cite{KanekoZagier})\label{conj:kanekozagier}
There exists a $\Q$-algebra isomorphism 
\begin{align*}
\varphi_{KZ}: \mathcal{Z}_{\mathcal{A}} &\longrightarrow \mathcal{Z}/\zeta(2)\mathcal{Z},\\
\zeta_{\mathcal{A}}(\kk) &\longmapsto \zeta_{\mathcal{S}}(\kk) \mod \zeta(2)\mathcal{Z}.
\end{align*}
\end{conjecture}
To give a new interpretation of this conjecture, we consider the $\Q$-vector space spanned by all $Z(\kk)$ 
\begin{align*}
\mathcal{Z}^{\cyc}=Z^{\star}(\h^{1})=Z(\h^{1}).  
\end{align*}
By Corollary \ref{cor:product} this is a $\Q$-subalgebra of $\mathcal{A}^{\cyc}$. 
The restriction of the map $\varphi: \mathcal{A}^\cyc \rightarrow \mathcal{A}$ defined in \eqref{eq:defphi} to $\mathcal{Z}^{\cyc}$ gives the surjective $\Q$-algebra homomorphism to the  $\Q$-algebra $\mathcal{Z}_{\mathcal{A}}$ of finite multiple zeta values denoted by
\begin{align*}
\varphi_{\mathcal{A}}: \mathcal{Z}^{\cyc} \longrightarrow \mathcal{Z}_{\mathcal{A}}\,.
\end{align*} 
For any index $\kk$ we have $\varphi_{\mathcal{A}}(Z(\kk))=\zeta_{\mathcal{A}}(\kk)$. 
On the other hand the relationship of the $Z(\kk)$ to the symmetric multiple zeta values is not understood yet, but we expect the following.
\begin{conjecture}\label{conj:mainconj}
\begin{enumerate}[i)]
\item There exists a $\Q$-algebra homomorphism 
\begin{align*}
\varphi_{\mathcal{S}}: \mathcal{Z}^\cyc &\longrightarrow \mathcal{Z}/\zeta(2)\mathcal{Z}\,,\\
Z(\kk) &\longmapsto \zeta_{\mathcal{S}}(\kk) \mod \zeta(2)\mathcal{Z} \,.
\end{align*} 
\item The equality $\ker \varphi_{\mathcal{S}} = \ker \varphi_{\mathcal{A}}$ holds.
\end{enumerate}
\end{conjecture}

This conjecture is a re-interpretation of the conjecture by Kaneko and Zagier.

\begin{theorem}
Conjecture \ref{conj:mainconj} implies Conjecture \ref{conj:kanekozagier}. 
\end{theorem}

We end this paper by giving some observation on the elements of the ideal $\ker \varphi_{\mathcal{A}}$ in $\mathcal{Z}^\cyc$. As an easy consequence of the definition of $\varphi_{\mathcal{A}}$ we obtain the following. 

\begin{proposition}
We have $\ker \varphi_{\mathcal{A}} =\mathcal{Z}^\cyc \cap \varpi \mathcal{A}^\cyc$, 
where $\varpi\mathcal{A}^\cyc$ denotes the ideal of $\mathcal{A}^\cyc$ generated by $\varpi$.
\end{proposition}

\begin{proof}
This is immediate from $\ker \varphi = \varpi \mathcal{A}^\cyc$.
\end{proof}
Lemma \ref{lem:product} implies that $\varpi \mathcal{Z}^{\cyc} \subset \mathcal{Z}^{\cyc}$. 
Hence $\varpi \mathcal{Z}^{\cyc} \subset \ker{\varphi_{\mathcal{A}}}$. 
However, we expect $ \varpi \mathcal{Z}^\cyc \neq \ker \varphi_{\mathcal{A}}$.
For example, by \cite[Theorem 7.1]{Hoffman} we have
\begin{align}
\zeta_{\mathcal{A}}(4,1)-2\,\zeta_{\mathcal{A}}(3,1,1)=0. 
\label{eq:hoffman-rel} 
\end{align}
Therefore $Z(4,1)-2Z(3,1,1) \in \ker{\varphi_{\mathcal{A}}}$, but it can be shown that $Z(4,1)-2Z(3,1,1) \notin \varpi \mathcal{Z}^\cyc$. So far it is not known how to describe the elements in 
$(\ker \varphi_{\mathcal{A}}) \backslash \varpi\mathcal{Z}^\cyc$ in general.


\end{document}